\documentclass[12pt]{amsart}

\usepackage[utf8]{inputenc} 
\usepackage{amssymb}
\usepackage{amsmath}
\usepackage{upref}
\usepackage{array}

\makeatletter
\@namedef{subjclassname@2010}{%
  \textup{2010} Mathematics Subject Classification}
\makeatother

\newtheorem{thm}{Theorem}[section]

\newtheorem{lem}[thm]{Lemma}
\newtheorem{prop}[thm]{Proposition}

\theoremstyle{definition}

\newtheorem*{xrem}{Remark}

\def\scr{\mathcal}
\def\e{\varepsilon}

\def\<{\langle}
\def\>{\rangle}

\def\R{\mathbb R}
\def\N{\mathbb N}

\def\e{\varepsilon}

\def\Var{\operatorname{Var}}

\begin{document}


\baselineskip=17pt
\sloppy

\title[LLN for Nonmeasurable RVs]{On the Law of Large Numbers for Nonmeasurable Identically Distributed Random Variables}
\author{Alexander R. Pruss}
\address{Department of Philosophy\\
Baylor University\\
One Bear Place \#97273\\
Waco, TX 76798-7273, USA}
\email{alexander\_pruss@baylor.edu}
\thanks{
{\em Bulletin of the Polish Academy of Sciences, Mathematics} (forthcoming), published by the Institute of Mathematics, Polish Academy of Sciences}


\begin{abstract}
Let $\Omega$ be a countable infinite product $\Omega^\N$ of copies of the same probability space $\Omega_1$, and let $\{ \Xi_n \}$ be the sequence of the coordinate projection functions from $\Omega$ to $\Omega_1$.  Let $\Psi$ be a possibly nonmeasurable function from $\Omega_1$ to $\R$, and let
$X_n(\omega) = \Psi(\Xi_n(\omega))$.  Then we can think of $\{ X_n \}$ as a sequence of independent but possibly nonmeasurable random variables on $\Omega$.
Let $S_n = X_1+\cdots+X_n$.  By the ordinary Strong Law of Large Numbers, we almost surely have $E_*[X_1] \le \liminf S_n/n \le \limsup S_n/n \le E^*[X_1]$, where $E_*$ and $E^*$ are the lower and upper expectations.  We ask if anything more precise can be said about the limit points of $S_n/n$ in the non-trivial
case where $E_*[X_1] < E^*[X_1]$, and obtain several negative
answers.  For instance, the set of points of $\Omega$ where $S_n/n$ converges is maximally nonmeasurable: it has inner measure zero and outer measure one.
\end{abstract}

\maketitle
\section{Introduction}
Ordinary random variables are $P$-measurable functions on a probability space $(\Omega,\scr F,P)$, where $\scr F$ is a $\sigma$-field on $\Omega$.
By the ordinary Strong and Weak Laws of Large Numbers (LLNs), if
$X_1,X_2,...$ are measurable identically distributed random variables with finite expectation, then
$(X_1+\cdots+X_n)/n \to E[X_1]$ almost surely (Strong Law) and in probability (Weak Law).  But we can also ask what happens to long-term means of samples when the random variables are not measurable.

Let $(\Omega,\scr F,P)$ be a probability space.  The following collects some known facts (see, e.g., \cite[Lemmas 1.2.2 and 1.2.3]{VW}) that allow
us to apply probabilistic techniques in the case of nonmeasurable random variables.

\begin{prop}\label{prop:L-U} Let $H$ be any subset of $\Omega$.  Then there are measurable sets $H_*$ and $H^*$ such that $H_* \subseteq H \subseteq H^*$ and
such that for any measurable $A\subseteq H$ we have $P(A) \le P(H_*)$ and for any measurable $B \supseteq H$ we have $P(B) \ge P(H^*)$.  The
sets $H_*$ and $H^*$ are uniquely defined up to sets of measure zero.

For any real-valued function $f$ on $\Omega$, there are measurable functions $f_*$ and $f^*$ such that $f_* \le f \le f^*$ everywhere and for any measurable $g$ on $\Omega$ such that $g\le f$ everywhere, we have $g\le f_*$ almost surely, while for any measurable $h$ on $\Omega$ such that $f \le h$ everywhere, we have $h\ge f^*$ almost surely.  The functions $f_*$ and $f^*$ are uniquely defined up to almost sure equality.
\end{prop}

The functions $f_*$ and $f^*$ are the {\em maximal measurable minorant} and {\em minimal measurable majorant} of $f$, respectively.

We then have $P_*(H)=P(H_*)$ and $P^*(H)=P(H^*)$, where $P_*$ and $P^*$ are the inner and outer measures generated by $P$.  Note that $H$ is measurable with respect to the completion of $P$ if and only if $P_*(H)=P^*(H)$, in which case it has that value as its measure with respect to the completion of $P$.

We say that a set is {\em maximally nonmeasurable}\/ provided that $P_*(H)=0$ and $P^*(H)=1$.  Such a set is one all of whose measurable subsets have null measure and all of whose measurable supersets have full measure.

As a replacement for the independence assumption in the case of ordinary random variables, take our probability space $(\Omega,\scr F,P)$ to be a product of the probability spaces $(\Omega_n,\scr F_n,P_n)$, and let our sequence of possibly nonmeasurable random variables be a sequence of functions $X_1,X_2,...$ on $\Omega$ such that $X_n(\omega_1,\omega_2,...)$
depends only on the value of $\omega_n$, so that there is a function $\Psi_n$ such that $X_n(\omega_1,\omega_2,...) = \Psi_n(\omega_n)$.  We will say that $X_1,X_2,...$ is then a sequence of independent identically-distributed possibly-nonmeasurable random variables (iidpnmrvs) providing that all the
probability spaces $(\Omega_n,\scr F_n,P_n)$ are the same space $(\Omega_1,\scr F_1,P_1)$ and that $\Psi_n$ is the same function $\Psi$ for all $n$.

The following fact about product measures and the $(\cdot)_*$ and $(\cdot)^*$ operators follows from \cite[Lemma~1.2.5]{VW}.

\begin{prop}\label{prop:product}
Suppose $(\Omega,\scr F,P)$ is a product of the probability spaces $(\Omega_n,\scr F_n,P_n)$ for $n=1,2,...$.  Let
$\Psi_n$ be a function on $\Omega_n$.  Let $X_n(\omega_1,\omega_2,...) = \Psi_n(\omega_n)$.
Let $Y_n(\omega_1,\omega_2,...) = (\Psi_n)_*(\omega_n)$ and $Z_n(\omega_1,\omega_2,...) = \Psi_n^*(\omega_n)$.
Then $P$-almost surely we have $(X_n)_* = Y_n$ and $X_n^* = Z_n$.
\end{prop}

In particular, if $X_1,...,X_n$ are iidpnmrvs, then $(X_1)_*,...,(X_n)_*$ are identically distributed independent random variables,
and so are $X_1^*,...,X_n^*$.  Let $S_n=X_1+\cdots+X_n$.  From the Strong Law of Large Numbers as applied to $\{(X_n)_*\}$ and $\{X_n^*\}$ (and using the fact that if $|X_1|^*$ is integrable, then so are $(X_1)_*$ and $X_1^*$) it then follows that almost surely:
\begin{equation}\label{eq:easy}
E[(X_1)_*] \le \liminf_{n\to\infty} \frac{S_n}{n} \le \limsup_{n\to\infty} \frac{S_n}{n} \le E[X_1^*].
\end{equation}
Here and elsewhere ``almost surely'' will mean {\em except perhaps on a set of probability zero}.  Thus an event holds almost surely provided its lower probability
is $1$.  (In the case of complete measures, this is equivalent to the usual notion of holding almost surely as holding on a set of full measure.)

We can define the {\em lower} and {\em upper expectations} of $X_1$ as $E_*[X_1]=E[(X_1)_*]$ and $E^*[X_1] = E[X_1^*]$, respectively (for more on lower
and upper expectations, see~\cite{HJ2002}).
Again, we have a trivial case when $E[(X_1)_*] = E[X_1^*]$ and then the Strong Law holds.  In that case, $(X_n)_* = X_n^*$ almost surely (since $(X_n)_* \le X_n \le X_n^*$), and $X_n$ will be measurable with respect to the completion of $P$.

The converse is also known~\cite{HJ1986}: if the Strong Law holds, then $X_1$ almost surely equals an $L^1$-function, and is measurable with respect to the completion of $P$.

Can \eqref{eq:easy} be improved on?  For instance, can the first or last almost sure inequality sometimes be
replaced by an equality?  Or can we say that in the non-trivial case it is almost surely true that $S_n/n$ diverges?  Our main result shows that the answers to these questions are negative, and that the failure of the Strong Law for non-measurable $X_1$ is radical.

\begin{thm}\label{th:main} Let $X_1,X_2,...$ be iidpnmrvs with $E^*[|X_1|] < \infty$.  Suppose $A$ is a non-empty proper subset of $[E_*[X_1],E^*[X_1]]$.  Then each of the following is maximally nonmeasurable:
\begin{enumerate}
\item[\textup{(i)}] the subset of $\Omega$ where $\liminf_{n\to\infty} S_n(n)/n$ is in $A$
\item[\textup{(ii)}] the subset of $\Omega$ where $\limsup_{n\to\infty} S_n(n)/n$ is in $A$
\item[\textup{(iii)}] the subset of $\Omega$ where $\lim_{n\to\infty} S_n(n)/n$ exists and is in $A$
\item[\textup{(iv)}] the subset of $\Omega$ where $\lim_{n\to\infty} S_n(n)/n$ exists
\item[\textup{(v)}] the subset of $\Omega$ where all the limit points of $S_n(n)/n$ are in $A$.
\end{enumerate}
\end{thm}
Thus in the non-trivial case ($[E_*[X_1],E^*[X_1]]$ has a non-empty proper subset if and only if $E_*[X_1]<E^*[X_1]$) nothing can be probabilistically said, with respect to $P$, about the limit points of $S_n(n)/n$ except that all the limit points lie within $[E_*[X_1],E^*[X_1]]$.

For completeness, here is a somewhat analogous result about the Weak Law:
\begin{thm}\label{th:weak} Let $X_1,X_2,...$ be iidpnmrvs  with $E^*[|X_1|] < \infty$.  Suppose $a \in [E[(X_1)_*],E[X_1^*]]$ and that $\e>0$ is sufficiently small that
$[E_*[X_1], E^*[X_1]]$ is not a subset of $[a-\e, a+\e]$.  Then $P_*(|S_n/n-a|>\e) \to 0$ and $P^*(|S_n/n-a|>\e) \to 1$ as $n\to\infty$.
\end{thm}

The proof of both theorems will be based on the following easy fact about the existence of extensions of measures.

\begin{lem}\label{lem:ext-f} Suppose $f$ is a function on a probability space $(\Omega, \scr F, P)$ and $f$ is simple, i.e., takes on only finitely many values.  Then there are extensions $P_*$ and $P^*$ of $P$ defined on the $\sigma$-field generated by $\scr F$ and $f$ such that $f = f_*$ almost surely with respect to $P_*$ and $f = f^*$ almost surely with respect to $P^*$.
\end{lem}

It suffices to see this where $f$ is an indicator function, and in that case the result follows from the observation that if $A \subseteq \Omega$ and $\alpha \in [P_*(A), P^*(A)]$, then there is an extension $P_\alpha$ of $P$ to $\sigma(\scr F\cup \{ A \})$ such that $P_\alpha(A)=\alpha$ (cf.\ \cite[p.~71]{Halmos}).

\begin{xrem} Our proofs of the theorems would be much simpler if we could have this for $f$ taking on infinitely many values, but alas the lemma is false in that case.  To see the falsity, let $\Omega$ be the open square $(0,1)^2$, make $\scr F$ be the
$\sigma$-field of subsets of the form $A \times (0,1)$ for $A\subseteq (0,1)$ Lebesgue-measurable, take $P$ to be the restriction of Lebesgue measure to $\scr F$, and set $f(x,y) = y$.  Then $f_* = 0$ almost surely with respect to $P$ but there is no extension of $P$ with respect to which $f=0$ almost surely, since
$f$ is nowhere equal to zero.
\end{xrem}

\section{Proofs}

We need some very easy preliminaries.

\begin{lem}\label{lem:equals}
Suppose that $f=g$ on a measurable set $B$.  Then $f_* = g_*$ and $f^* = g^*$ almost surely on $B$.
\end{lem}

\begin{proof}
Let $h(\omega) = f^*(\omega)$ for $\omega \notin B$ and let $h(\omega) = \min(f^*(\omega), g^*(\omega))$ otherwise.
Then $h$ is a measurable function such that $h\ge f$.  Hence $h \ge f^*$ almost surely.  Then
$g^* \ge h \ge f^*$ almost surely on $B$.  In the same way, we see that $f^* \ge g^*$ almost surely on $B$.
That $f_*=g_*$ almost surely on $B$ is proved the same way.
\end{proof}

The following is a simple consequence of~\cite[Lemma 1.2.2]{VW}.

\begin{lem}\label{lem:approx}
If $|f-g| \le \e$ everywhere, then $|f_*-g_*| \le \e$ and $|f^*-g^*| \le \e$ almost surely.
\end{lem}

The following trivial lemma encapsulates the strategy for the proof of our theorems.

\begin{lem}\label{lem:strategy}
Suppose that $B$ is subset of a probability space $(\Omega,\scr F,P)$ such that there are extensions
$P_1$ and $P_2$ of $P$ so that $B$ is $P_1$- and $P_2$-measurable with $P_1(B)=x_1$ and $P_2(B)=x_2$.
Then $P_*(B)\le x_1$ and $x_2\le P^*(B)$.  In particular, if $P_1(B)=0$ and $P_2(B)=1$, then $B$ is maximally nonmeasurable.
\end{lem}

\begin{proof}
We have $P_*(B) = P(B_*) = P_1(B_*) \le P_1(B) = x_1$ and $P^*(B) = P(B^*) = P_2(B^*) \ge P_2(B) = x_2$,
where $B_*$ and $B^*$ are defined with respect to $P$.
\end{proof}

By Lemma~\ref{lem:strategy}, we need to show that for each of the subsets of $\Omega$ mentioned in Theorem~\ref{th:main}, there is an extension of $P$ that assigns measure zero to the subset and another that assigns it measure one.

Finally, as we will soon see, the following lemma will yield all the results of Theorem~\ref{th:main}.

\begin{lem}\label{lem:conv}
Suppose $X_1,X_2,...$ is a sequence of iidpnmrvs such that $E[|X_1|^*]<\infty$.  Then for any $\alpha \in [E_*[X_1],E^*[X_1]]$ there is
an extension $P'$ of $P$ such that $P'$-almost surely $S_n/n$ converges to $\alpha$.
If $E_*[X_1] < E^*[X_1]$ then there is an extension $P''$ of $P$ such that $P''$-almost surely $S_n$ diverges.
\end{lem}

\begin{proof}
Write $E_P[f]$ for the expectation of $f$ with respect to $P$, i.e., $\int_\Omega f(\omega)\, dP(\omega)$.
The variables $X_1,X_2,...$ are defined by $X_n(\omega_1,\omega_2,...) = \Psi(\omega_n)$ on our product space $\Omega$ for some
real-valued function $\Psi$.

Let $X_n' = X_n \cdot 1_{\{ |X_n| \le n \}}$.

Let $A_{n,\pm} = \{ \pm X_n > n \}$.  Let $A_n = \{ |X_n| > n \} = A_{n,+} \cup A_{n,-}$.
Observe that $A_{n,+} \subseteq \{ X_n^* > n \}$ and $A_{n,-} \subseteq \{ (X_n)_* < -n \}$.
Let $B_n = \{ |X_n^*| > n \} \cup \{ |(X_n)_*| > n \}$.  Clearly $A_n \subseteq B_n$.
Then:
\begin{equation}\label{eq:sum-Bn}
   \sum_{n=1}^\infty P(B_n) \le \sum_{n=1}^\infty P(|(X_n)_*|>n) + \sum_{n=1}^\infty P(|X_n^*|>n) < \infty,
\end{equation}
since $\sum_{n=1}^\infty P(|X|>n) \le \int_0^\infty P(|X|>t) \, dt = E[|X|]$, and since both $(X_n)_*$ and $X_n^*$ have
finite expectations given that $|X_n|^*$ does.  In particular, almost surely, only finitely many of the
$B_n$ occur by Borel-Cantelli.  Since $A_n\subseteq B_n$, almost surely only finitely many of the $A_n$ occur.

I now claim that $E[(X_n')_*] \to E[(X_1)_*]$ and $E[(X_n')^*] \to E[X_1^*]$ as $n\to\infty$.  We only need to prove the latter
claim since the former follows by applying the latter to the iidpnmrv sequence $\{ -X_n \}$.  Now, outside of $B_n$, we have
$X_n' = X_n$, and since $B_n$ is measurable it follows from Lemma~\ref{lem:equals} that outside of $B_n$ we
have $(X_n')^* = X_n^*$ almost surely.  Moreover, everywhere on $B_n$ we have $X_n' = 0$ and hence $(X_n')^* = 0$ by Lemma~\ref{lem:equals}.
Thus:
\begin{align*}
|E[(X_n')^*] - E[X_1^*]| &= |E[(X_n')^*] - E[X_n^*]| \\
                         &\le E[|(X_n')^*-X_n^*|] \\
                         &\le E[|(X_n')^*-X_n^*| \cdot 1_{B_n}] \\
                         &= E[|X_n^*| \cdot 1_{B_n}] \\
                         &\le E[|X_n^*| \cdot 1_{ \{ |X_n|^* > n \} }] \\
                         &= E[|X_1^*| \cdot 1_{ \{ |X_1|^* > n \} }] \to 0,
\end{align*}
since $E[|X_1|^*] < \infty$.

Let $\Psi'_n = \Psi \cdot 1_{ \{ |\Psi| \le n \}}$ so that $X_n'(\omega_1,\omega_2,...) = \Psi'_n(\omega_n)$.  Choose a simple function $\Psi_{1,n}$ such that both $|\Psi_{1,n}-\Psi'_n|\le 1/n$ and $|\Psi_{1,n}|\le n$ everywhere
on $\Omega_1$.

By Lemma~\ref{lem:ext-f} there is an extension $P_{1,n,0}$ of $P_1$ such that $\Psi_{1,n}$ is $P_{1,n,0}$-measurable and
$P_{1,n,0}$-almost surely $\Psi_{1,n} = (\Psi_{1,n})_*$, and an extension $P_{1,n,1}$ of $P_1$ such that $\Psi_{1,n}$ is $P_{1,n,1}$-measurable and
$P_{1,n,1}$-almost surely $\Psi_{1,n} = \Psi_{1,n}^*$.  Let $\scr F_{1,n}$ be the $\sigma$-field on $\Omega_1$ generated by $\scr F_1$ and
$\Psi_{1,n}$.

For $i=0,1$, let $P^i$ be the product of the measures $P_{1,1,i}, P_{1,2,i}, P_{1,3,i}, ...$.  This is an extension of $P$.

Let $Y_n(\omega_1,\omega_2,...) = \Psi_{1,n}(\omega_n)$.  By Proposition~\ref{prop:product}, $(Y_n)_*(\omega_1,\omega_2,...) = (\Psi_{1,n})_*(\omega_n)$ for $P$-almost all $(\omega_1,\omega_2,...)$.  Then $P^0$-almost surely we have $Y_n = (Y_n)_*$ and $P^1$-almost surely $Y_n = Y_n^*$ (where $(Y_n)_*$ and $Y_n^*$ are
defined with respect to $P$).
Define the measure $P' = (1-a)P^0 + aP^1$, where $a$ is such that $\alpha = (1-a)E_P[(X_1)_*] + aE_P[X_1^*]$.
Then $E_{P'}[Y_n] = (1-a)E_P[(Y_n)_*]+aE_P[Y_n^*]$.
Now $|(Y_n)_*-(X_n')_*| \le 1/n$ and $|Y_n^*-(X_n')^*| \le 1/n$ everywhere by Lemma~\ref{lem:approx} and the choice of $\Psi_{1,n}$, so by uniform convergence we have
$E[(Y_n)_*-(X_n')_*]$ and $E[Y_n^*-(X_n')^*]$ converging to zero.  Moreover, we have already seen that
$E[(X_n')_*]\to E[(X_1)_*]$ and $E[(X_n')^*]\to E[X_1^*]$, so $E_P[(Y_n)_*] \to E_P[(X_1)_*]$ and $E_P[Y_n^*] \to E_P[X_1^*]$.  Thus, $E_{P'}[Y_n] \to (1-a)E_P[(X_1)_*]+aE_P[X_1^*]=\alpha$.

Now $|Y_n-X_n'| \le 1/n$ everywhere and  $|X_n'| \le \max(|(X_n)_*|, |X_n^*|) \cdot 1_{B_n}$.  Thus:
\begin{align*}
 \sum_{n=1}^\infty &\frac{\Var_{P'}[Y_n]}{n^2} \le  \sum_{n=1}^\infty \frac{E[Y_n^2]}{n^2} \\
    &= \sum_{n=1}^\infty \frac{E_{P'}[((X_n'+(Y_n-X_n'))^2]}{n^2} \\
    &\le 2\sum_{n=1}^\infty \frac{E_{P'}[(X_n')^2]}{n^2} + \sum_{n=1}^\infty \frac{2}{n^3} \\
    &= 2\sum_{n=1}^\infty \frac{E_{P'}[(X_n)_*)^2\cdot 1_{B_n}]}{n^2} + 2\sum_{n=1}^\infty \frac{E_{P'}[(X_n^*)^2\cdot 1_{B_n}]}{n^2} + \sum_{n=1}^\infty \frac{2}{n^3} \\
    &= 2\sum_{n=1}^\infty \frac{E_{P}[((X_n)_*)^2\cdot 1_{B_n}]}{n^2} + 2\sum_{n=1}^\infty \frac{E_{P}[(X_n^*)^2\cdot 1_{B_n}]}{n^2} + O(1) \\
    &\le 2\sum_{n=1}^\infty \frac{E_{P}[((X_n)_*)^2\cdot 1_{\{ |(X_n)_*| > n\}}]}{n^2} + 2\sum_{n=1}^\infty \frac{E_{P}[(X_n^*)^2\cdot 1_{\{ |X_n^*| > n\}}]}{n^2} 
         + O(1),
\end{align*}
where the last equality follows from the fact that $P$ is an extension of $P'$ and $(X_n)_*$ and $X_n^*$ are $P$-measurable.
The finiteness of the right hand side then follows from the proof of \cite[Lemma~2.4.3]{Durrett}.
It then follows from Kolmogorov's Strong Law~\cite[Theorem 5.8]{MW} that $P'$-almost
surely $(Y_1+\cdots+Y_n - E[Y_1+\cdots+Y_n])\to 0$ and hence $(Y_1+\cdots+Y_n)/n \to \alpha$.

Let $T_n = Y_1+\cdots+Y_n$.
Let $S_n' = X_1'+\cdots+X_n'$.
Since $|X_n'-Y_n| \le 1/n$ everywhere, $|S_n'/n - T_n/n| \le (1/n)(1+1/2+\cdots+1/n)$
and the right hand side converges to zero.  Moreover, $P$-almost surely for all but finitely many $n$ we have $X_n = X_n'$, so
that $P$-almost surely $(S_n-S_n')/n \to 0$.  Thus, $P$-almost surely we have $(S_n-T_n)/n \to 0$.  But $P'$ extends $P$, so
this also holds $P'$-almost surely.  But since $P'$-almost surely $T_n/n \to \alpha$, we also have $S_n/n$ converging $P'$-almost surely to $\alpha$.

To construct $P''$, first recall that $E_P[(Y_n)_*] \to E_P[(X_1)_*]$ and $E_P[Y_n^*] \to E_P[X_1^*]$.
Let $\gamma = E_P[(X_1)_*]$ and $\delta = E_P[X_1^*]$.
For $n>0$, let $a_n$ be a strictly increasing sequence of positive integers such that (a)~$|E_P[(Y_k)_*] - E_P[(X_1)_*]|<1/n$ and
$|E_P[Y_k^*] - E_P[X_1^*]| < 1/n$ for all $k \ge a_n$, and (b)~$a_{n+1}/a_n \to \infty$.
Let $a_0 = 1$.

Let $L_n = \{ a_{n-1}+1,a_{n-1}+2,...,a_n \}$.  Let $L = L_1 \cup L_3 \cup L_5 \cup ...$.
For $n \in L$, let $\alpha_n = E_P[(Y_n)_*]$ and for $n\notin L$, let $\alpha_n = E_P[Y_n^*]$.
Let $\beta_n = \alpha_1+\cdots+\alpha_n$.
I now claim that $\beta_{a_{2n}} / a_{2n} \to \delta$ and $\beta_{a_{2n+1}} / a_{2n+1} \to \gamma$ as $n\to\infty$.
For, since both $E_P[(Y_n)_*]$ and $E_P[Y_n^*]$ converge, there is an $M<\infty$ such that $|\alpha_k|\le M$ for all $K$, and then by the choice of $a_n$:
\begin{align*}
    \beta_{a_{2n}} - a_{2n}\delta &= \sum_{k={a_{2n-1}+1}}^{a_{2n}} (E_P[Y_k^*]-\delta) - a_{2n-1} \delta + \sum_{k=1}^{a_{2n-1}} \alpha_k \\
        &= O(a_{2n}/(2n-1)) + O(a_{2n-1}).
\end{align*}
Since $a_{2n-1}/a_{2n} \to 0$ by condition (b) above, we see that $\beta_{a_{2n}}/a_{2n}-\delta$ converges to $0$ as desired.
Likewise:
\begin{align*}
    \beta_{a_{2n+1}} - a_{2n+1}\gamma &= \sum_{k={a_{2n}+1}}^{a_{2n+1}} (E_P[(Y_k)_*]-\gamma) - a_{2n} \gamma + \sum_{k=1}^{a_{2n}} \alpha_k \\
        &= O(a_{2n+1}/2n) + O(a_{2n}),
\end{align*}
and so $\beta_{a_{2n+1}} / a_{2n+1} \to \gamma$.

Now define $P_{1,n}''$ to be equal to $P_{1,n,0}$ if $n\in L$, and to $P_{1,n,1}$ if $n\notin L$.  Let $P''$ be the product of the
measures $P_{1,1}'',P_{1,2}'',...$.

In exactly the same way as we proved above using Kolmogorov's Strong Law that $P'$-almost surely $(T_n-E_{P'}[T_n])$, we can also
show that $P''$-almost surely $(T_n-E_{P''}[T_n])/n \to 0$.
But $E_{P''}[T_n] = \beta_n$ since $E_{P''}[Y_n] = E_P[(Y_n)_*]$ if $n\in L$ and $E_{P''}[Y_n] = E_P[Y_n^*]$ otherwise.
Thus, $P''$-almost surely $T_{a_{2n}}/a_{2n} \to \delta$ and $T_{a_{2n+1}}/a_{2n+1} \to \gamma$.

But we have already seen that $(S_n/n - T_n)/n$ converges to zero $P$-almost surely, and hence also $P''$-almost surely.
Thus $P''$-almost surely $S_{a_{2n}}/a_{2n} \to \delta$ and $S_{a_{2n+1}}/a_{2n+1} \to \gamma$.  Since $\delta>\gamma$, our desired
divergence result follows.
\end{proof}

\begin{proof}[Proof of Theorem~\ref{th:main}]
Choose any $a_1 \in A$.
By Lemma~\ref{lem:conv}, we have an extension $P'$ of $P$ such that $P'$-almost surely $S_n/n$ converges to $a_1$.
Moreover, since $A$ is a proper non-empty subset of $[E_P[(X_1)_*],E_P[X_1^*]]$, the latter interval must contain at least
two points and hence $E_P[(X_1)_*] < E_P[X_1^*]$, so by the same lemma there is an extension $P''$ of $P$ such that $P''$-almost surely
$S_n/n$ diverges.  Now choose any $a_2 \in [E_P[(X_1)_*],E_P[X_1^*]]-A$.  Again, by the same lemma there is an extension $P'''$ of
$P$ such that $P'''$-almost surely $S_n/n$ converges to $a_2$.

All the events described in (i)-(v) will happen whenever $S_n/n \to a_1$, and so they all have $P'$-measure $1$.
Events (i), (ii), (iii) and (v) cannot happen when $S_n/n \to a_2$, and so they all have $P'''$-measure $0$.  And
event (iv) as $P''$-measure $0$.  Thus, each event has measure $1$ under one extension of $P$ and measure $0$ under
some other extension, and so by Lemma~\ref{lem:strategy}, each event is maximally $P$-nonmeasurable.
\end{proof}

\begin{proof}[Proof of Theorem~\ref{th:weak}]
Let $a$ and $\e$ be as in the statement of the theorem.  The conditions of the theorem guarantee that there is a point $b \in [E_*[X_1], E^*[X_1]]-[a-\e,a+\e]$.
Let $\gamma = |b-a|-\e$.  Since $b \notin [a-\e,a+\e]$, we have $\gamma>0$.

By Lemma~\ref{lem:conv}, let $P'$ be an extension of $P$ such that $P'$-almost surely $S_n/n$ converges to $a$, and
let $P''$ be an extension of $P$ such that $P''$-almost surely $S_n/n$ converges to $b$.
Then $S_n/n$ converges to $a$ in $P'$-probability and to $b$ in $P''$-probability.  Hence
$\lim_{n\to\infty} P'(|S_n/n - a|>\e) = 0$.  By Lemma~\ref{lem:strategy}, we have $\lim_{n\to\infty} P_*(|S_n/n - a|>\e) = 0$.
Moreover, $0 = \lim_{n\to\infty} P''(|S_n/n - b|>\gamma) \ge \limsup_{n\to\infty} P''(|S_n/n - a|\le \e)$ by choice of $\gamma$.
Thus $P''(|S_n/n - a| \le \e)$ converges to 0, and so $P''(|S_n/n - a| > \e)$ converges to $1$, so that $P^*(|S_n/n - a| > \e)$ also
converges to $1$.
\end{proof}

\section*{Acknowledgments}
The author would like to thank Arthur Paul Pedersen for a number of discussions and an anonymous referee for a number of helpful suggestions.

\end{document}